\documentclass[10pt]{amsart}

\usepackage{amssymb}
\usepackage{hyperref}
\usepackage{amsmath,amscd}
\usepackage[margin=1in]{geometry}
\usepackage{multicol}

\newtheorem{theorem}{Theorem}[section]

\newtheorem{proposition}[theorem]{Proposition}
\newtheorem{corollary}[theorem]{Corollary}

\theoremstyle{definition}
\newtheorem{definition}[theorem]{Definition}
\newtheorem{example}[theorem]{Example}

\theoremstyle{remark}
\newtheorem{remark}[theorem]{Remark}

\usepackage{array}
\newcolumntype{C}{>{$}c<{$}} 

\def\pc{\mathrm{pc}}
\def\ac{\mathrm{ac}}
\def\rac{\mathrm{rac}}
\def\rpc{\mathrm{rpc}}

\begin{document}

\title{Partially Palindromic Compositions} 
\author{Jia Huang}
\address{Department of Mathematics and Statistics, University of Nebraska, Kearney, NE 68849, USA}
\curraddr{}
\email{huangj2@unk.edu}

\begin{abstract}
We generalize recent work of Andrews, Just, and Simay on modular palindromic compositions and anti-palindromic compositions by viewing all compositions partially (modular) palindromic or anti-palindromic.
More precisely, we enumerate compositions by the extent to which they are (modular) palindromic or anti-palindromic.
We obtain various closed formulas from generating functions and provide bijective proofs for many of them. 
We recover some known results of Andrews, Just, and Simay and discover new connections with numerous sequences in The On-Line Encyclopedia of Integer Sequences (OEIS). 
\end{abstract}


\maketitle

\section{Introduction}

The concept of (integer) composition is useful and well studied in enumerative and algebraic combinatorics.
A \emph{composition} of $n$ is a sequence $\alpha=(\alpha_1, \ldots, \alpha_\ell)$ of positive integers with $|\alpha|:=\alpha_1+\cdots+\alpha_\ell=n$.
The \emph{parts} of $\alpha$ are $\alpha_1,\ldots,\alpha_\ell$ and the \emph{length} of $\alpha$ is $\ell(\alpha):=\ell$. 
We often drop the parentheses and commas in $\alpha$ when the parts are all single-digit numbers.
We can encode $\alpha$ in a binary string of length $n$ whose $i$th entry is $1$ if $i\in \{\alpha_1, \alpha_1+\alpha_2, \ldots, \alpha_1+\cdots+\alpha_\ell\}$ or $0$ otherwise. 
For instance, $\alpha=24112$ is a composition of $10$ corresponding to the binary string $0100011101$.
Thus there are exactly $2^{n-1}$ compositions of $n$ as they are in bijection with binary strings of length $n$ whose last entry must be $1$.

It is well known that the number of palindromic compositions of $n$ is $2^{\lfloor n/2 \rfloor}$~\cite{PalComp}, where a composition $\alpha=(\alpha_1, \ldots, \alpha_\ell)$ is \emph{palindromic} if $\alpha_i=\alpha_{\ell+1-i}$ for all $i=1, 2, \ldots, \lfloor \ell/2 \rfloor$.
Recently, Andrews and Simay~\cite{ParityPalComp} generalized this to \emph{parity palindromic compositions} by replacing the condition $\alpha_i=\alpha_{\ell+1-i}$ with $\alpha_i \equiv \alpha_{\ell+1-i} \pmod 2$ and Just~\cite{PalCompMod} further generalized it to \emph{palindromic compositions modulo $m$} by imposing the condition $\alpha_i \equiv \alpha_{\ell-i} \pmod m$ for an arbitrary positive integer $m$.
By Just~\cite{PalCompMod}, the generating function of the number $\pc(n,m)$ of palindromic compositions of $n$ modulo $m$ is
\[ \sum_{n\ge1} \pc(n,m) q^n = \frac{q+2q^2-q^{m+1}}{1-2q^2-q^m}. \]
For $m=2$ and $n\ge1$ we have $\pc(2n,2) = \pc(2n+1,2) = 2\cdot 3^{n-1}$ with an analytic proof by Andrews and Simay~\cite{ParityPalComp}, a combinatorial proof by Just~\cite{PalCompMod}, and a recursive proof by Vatter~\cite{RecursiveParity}.
For $m=3$, Just~\cite{PalCompMod} showed, both analytically and combinatorially, that $\pc(1,3)=1$ and $\pc(n,3) = 2 F_{n-1}$ for all $n\ge2$, where $F_n$ is the ubiquitous \emph{Fibonacci number} defined by $F_0=0$, $F_1=1$, and $F_n=F_{n-1}+F_{n-2}$ for $n\ge2$.
The case $m>3$ was also briefly discussed by Just~\cite{PalCompMod}.

On the other hand, Andrews, Just and Simay~\cite{AntiPalComp} defined a composition $\alpha=(\alpha_1, \ldots, \alpha_\ell)$ to be \emph{anti-palindromic} if $\alpha_i \ne \alpha_{\ell-i}$ for all $i=1,2, \ldots, \lfloor \ell/2 \rfloor$.
They showed, among other things, that the number of anti-palindromic compositions of $n$ is $\ac(n) = T_n + T_{n-2}$, where $T_n$ is a \emph{tribonacci number} defined by the recurrence $T_n = T_{n-1} + T_{n-2} + T_{n-3}$ for $n\ge 3$ with initial conditions $T_i = 0$ for $i<1$ and $T_1 = T_2 = 1$~\cite[A000073]{OEIS}.

In this paper we generalize the aforementioned work~\cite{ParityPalComp, AntiPalComp, PalComp, PalCompMod} by viewing all compositions partially (modular) palindromic/anti-palindromic.
More precisely, we count compositions by the extent to which they are (modular) palindromic/anti-palindromic.
This is analogous to  the well-known Euler's partition theorem generalized to Glaisher's theorem and further to Franklin's theorem (for more details and a parallel theory on compositions, see recent work~\cite{CompRes}).
 
\begin{definition}\label{def:pc}
(i) For $n,k\ge0$, let $\pc^k(n)$ denote the number of compositions $\alpha=(\alpha_1, \ldots, \alpha_\ell)$ of $n$ satisfying $\#\{ 1\le i\le \ell/2: \alpha_i \ne \alpha_{\ell+1-i}\} = k$; note that $\pc^0(0)=1$ and $\pc^k(0)=0$ for $k\ge1$.
We further introduce a refinement $\pc^k(n) =  \pc_+^k(n) + \pc_-^k(n)$: among all compositions $\alpha$ counted by $\pc^k(n)$, those with $\ell$ even (i.e., having no middle part) or with $\ell$ odd and $\alpha_{(\ell+1)/2}$ even (i.e., having an even middle part) are counted by $\pc_+^k(n)$ and those with $\ell$ odd and $\alpha_{(\ell+1)/2}$ odd (i.e., having an odd middle part) are counted by $\pc_-^k(n)$.
For example, $\pc^1_+(4) = |\{31,13\}|=2$, $\pc^1_-(4) = \{211, 112\}|=2$, and $\pc^1(4)=2+2=4$.

\noindent
(ii) We define $\ac^k(n)$, $\ac_+^k(n)$, and $\ac_-^k(n)$ by replacing $\alpha_i \ne \alpha_{\ell+1-i}$ with $\alpha_i = \alpha_{\ell+1-i}$ in (i).

\noindent
(iii) For any positive integer we define $\pc^k(n,m)$, $\pc_+^k(n,m)$, and $\pc_-^k(n,m)$ by replacing $\alpha_i \ne \alpha_{\ell+1-i}$ with $\alpha_i \not\equiv \alpha_{\ell+1-i} \pmod m$ in (i), and define $\ac^k(n,m)$, $\ac_+^k(n,m)$, and $\ac_-^k(n,m)$ by using $\equiv$ instead of $\not\equiv$.

\noindent (iv)
We also define $\pc^k(n,\infty):=\pc^k(n)$, $\pc^k_+(n,\infty):=\pc_+^k(n)$, $\pc^k_-(n,\infty):=\pc^k_-(n)$, $\ac^k(n,\infty):=\ac^k(n)$, $\ac^k_+(n,\infty):=\ac_+^k(n)$, and $\ac^k_-(n,\infty):=\pc^k_-(n)$.
We often drop the superscript $k$ when $k=0$ since $\ac^0(n)$ and $\pc^0(n,m)$ agree with $\ac(n)$~\cite{AntiPalComp} and $\pc(n,m)$~\cite{PalCompMod}, respectively.
\end{definition}

By Definition~\ref{def:pc}, we have $\pc^k(n,m) = \pc_+^k(n,m) + \pc_-^k(n,m)$ and $\pc_-^k(0,m) = 0$.
We also have $\pc_+^k(n,m) = \pc_-^k(n+1,m)$ since any composition $\alpha=(\alpha_1, \ldots, \alpha_\ell)$ counted by $\pc_+^k(n,m)$ corresponds to a composition $\alpha'$ counted by counted by $\pc_-^k(n+1, m)$, where $\alpha':=(\alpha_1, \ldots, \alpha_s, 1, \alpha_{s+1}, \ldots, \alpha_\ell)$ if $\ell=2s$, or $\alpha':=(\alpha_1, \ldots, \alpha_s, \alpha_{s+1}+1, \alpha_{s+2}, \ldots, \alpha_\ell)$ if $\ell=2s+1$ and $\alpha_{s+1}$ is even. 
It follows that
\begin{equation}\label{eq:pc+-}
\pc^k(n,m) = \pc_+^k(n,m) + \pc_+^k(n-1,m)
\end{equation}
where $\pc_+^k(-1,m):=0$. 
Similarly, with $\ac_+^k(-1,m):=0$ we have
\begin{equation}\label{eq:ac+-}
\ac^k(n,m) = \ac_+^k(n,m) + \ac_+^k(n-1,m).
\end{equation}
To obtain closed formulas for $\pc^k(n,m)$ and $\ac^k(n,m)$, it suffices to do that for $\pc_+^k(n,m)$ and $\ac_+^k(n,m)$.

In Section~\ref{sec:ppc} we obtain one closed formula for $\pc_+^k(n)$ and three for $\ac_+^k(n)$ using their generating functions.
We also combinatorially prove all of these formulas except the last one for $\ac_+^k(n)$ via a bijection between compositions and certain pairs of nonnegative integer sequences.

For $k=0$, our formula for $\pc^k_+(n)$ becomes $\pc_+(n) = 2^{n/2}$ if $n$ is even or $\pc_+(n)=0$ if $n$ is odd, and this implies the well-known fact that $\pc(n) = 2^{\lfloor n/2\rfloor}$~\cite{PalComp}.
We also have $\pc_+^1(n) = 2+(\lceil n/2 \rceil-2)2^{\lceil n/2 \rceil}$ for $n\ge0$~\cite[A036799]{OEIS}.
On the other hand, when $k=0$ the three formulas for $\ac_+^k(n)$ all reduce to a tribonacci number $T'_{n+1}$ defined by the recurrence $T'_n=T'_{n-1}+T'_{n-2}+T'_{n-3}$ for $n\ge3$ with initial conditions $T'_0=0$, $T'_1=1$, and $T'_2=0$~\cite[A001590]{OEIS}.
Combining this with~\eqref{eq:ac+-} gives $\ac(n) = T'_{n+1}+T'_n$ which is different from the formula $\ac(n) = T_{n}+T_{n-2}$ obtained by Andrews, Just and Simay~\cite{AntiPalComp}, although they are equivalent via the relations $T'_{n+1}=T_{n-1}+T_{n-2}$ and $T'_n=T_{n}-T_{n-1}$~\cite[A001590]{OEIS}.

We still have a fourth formula for $\ac^k(n)$, which is an alternating sum.
We suspect that it has a combinatorial proof via inclusion-exclusion. 
For $k=0$ this formula reduces to $\ac(n)=T_{n+1}-T_{n-1}$, which is also equivalent to $\ac(n) = T'_{n+1}+T'_n$.
A byproduct of this is a formula for the tribonacci number $T_n$ which is not found in OEIS~\cite[A000073]{OEIS} but has a simple bijective proof.

In Section~\ref{sec:rpc} we extend the concept of reduced anti-palindromic compositions introduced by Andrews, Just and Simay~\cite{AntiPalComp} by defining $\rpc^k(n)$ (resp., $\rac^k(n)$) as the number of equivalence classes of compositions counted by $\pc^k(n)$ (resp., $\ac^k(n)$) under swaps of the first and last parts, the second and second last parts, and so on.
We obtain closed formulas for the similarly defined $\rpc^k_+(n)$ and $\rac^k_+(n)$, which lead to formulas for $\rpc^k(n)$ and $\rac^k(n)$.
It turns out that $\rpc^k(n)=\pc^k(n)/2^k$ and $\rac_+^k(n)$ coincides with the number of compositions of $n-k$ with exactly $k$ parts equal to $1$~\cite[A105422]{OEIS}.  

Now let $m$ be a positive integer.
In Section~\ref{sec:pcm} we provide a positive sum formula and an alternating sum formula for $\pc_+^k(n,m)$, and thus for $\pc^k(n,m)$ by~\eqref{eq:pc+-}.
In particular, we have a formula for $\pc(n,m)$, which implies the formulas of $\pc(n,2)$ and $\pc(n,3)$ as well as some properties of $\pc(n,m)$ given by Just~\cite{PalCompMod}.
In section~\ref{sec:rpcm} we provide formulas for $\rpc_+^k(n,m)$ and $\rpc^k(n,m)$ (defined similarly as $\rpc_+^k(n)$ and $\rpc^k(n)$) and find connections to certain sequences in OEIS~\cite{OEIS} for special values of $m$ and $k$.

In Section~\ref{sec:acm} we provide a positive sum formula and an alternating sum formula for $\ac_+^k(n,m)$ as well as a third formula for $\ac^k(n,m)$ which is also an alternating sum.
For $k=0$ we find no connection of $\ac_+(n,m)$ and $\ac(n,m)$ with existing sequences in OEIS~\cite{OEIS}, but for $m=1$, a signed version of $\ac_+^k(n,1)$ is an interesting Riordan array~\cite[A158454]{OEIS} and $\ac^k(n,1)$ counts compositions of $n$ with $2k$ or $2k+1$ parts.
In Section~\ref{sec:racm} we provide formulas for $\rac_+^k(n,m)$ and $\rac^k(n,m)$ (defined similarly as $\rac_+^k(n)$ and $\rac^k(n)$) and find connections to various integer sequences in OEIS~\cite{OEIS} for small values of $m$ and $k$.

Below is a summary of the integer sequences in OEIS~\cite{OEIS} that occur in this paper. 
\begin{itemize}
\item 
A036799$(n) = \pc_+(2n+1)=\pc_+(2n+2)$, A025192$(n) = \pc_+(2n,2)$, A008346$(n) = \pc_+(n+2,3)/2$
\item
A001590$(n+1) = \ac_+(n)$, 
A000073$(n+2)-$
A000073$(n) = \ac(n)$
\item
A212804$(n) = \rac_+(n)$,
A006367$(n) = \rac_+^1(n+2)$,
A105423$(n) = \rac_+^2(n+4)$, \\
A105422$(n,k) = \rac_+^k(n+k)$,
A324969$(n) = \rac(n)$,
A208354$(n) = \rac^1(n+2)$
\item
A002620$(n) = \ac_+^1(n,1)$, A001752$(n) = \ac^2_+(n+4,1)$, A001769$(n) = \ac^3_+(n+6,1)$, \\
A001780$(n) = \ac^4_+(n+8,1)$, A001786$(n) = \ac^5_+(n+10,1)$,
A161680$(n) = \ac^1(n,1)$, \\
A000332$(n) = \ac^2(n,1)$, A000579$(n) = \ac^3(n,1)$, A000581$(n) = \ac^4(n+8,1)$
\item
A052547$(n)=\rpc_+(n,1)$, A001870$(n) = \rpc_+^1(2n+3,2)$, 052534$(n)=\rpc_+(2n,4)$, \\
A028495$(n)=\rpc(n,1)$, A094967$(n)=\rpc(n,2)$
\item
A062200$(n) = \rac_+(n,2)$,
A113435$(n)=\rac(n,3)$, \\
A008805$(n)=\rac_+^1(n+2,1)$, 
A096338$(n)=\rac_+^2(n+3,1)$,
A299336$(n)=\rac_+^3(n+6,1)$, \\
A002620$(n)=\rac^1(n,1)$,
A002624$(n)=\rac^2(n+4,1)$,
A060099$(n)=\rac^3(n+6,1)$, \\
A060100$(n)=\rac^4(n+8,1)$,
A060101$(n)=\rac^5(n+10,1)$,
A060098$(n,k)=\rac^k(n+2k,1)$
\end{itemize}

All of the closed formulas in this paper are derived from generating functions and they are either positive sums or alternating sums of products of binomial coefficients.
To be precise, we define the binomial coefficient
\[ \binom{a}{b} := 
\begin{cases}
\frac{a!}{b!(a-b)!}  & \text{if } a\ge b\ge 1 \\
1 & \text{if } b=0 \\
0 & \text{otherwise}
\end{cases} \]
for integers $a$ and $b$.
While we have bijective proofs for most of the positive sum formulas in this paper, it would be nice to see combinatorial proofs for the rest (possibly via inclusion-exclusion for the alternating sum formulas).

\tableofcontents

\section{Partially palindromic/anti-palindromic compositions}\label{sec:ppc}

In this section we provide closed formulas for $\pc^k(n)$ and $\ac^k(n)$.  
By~\eqref{eq:pc+-} and~\eqref{eq:ac+-}, it suffices to do this for $\pc_+^k(n)$ and $\ac_+^k(n)$.

\begin{theorem}\label{thm:pc}
For any integers $n,k\ge0$ we have
\begin{align} \label{eq:pc}
\pc_+^k(n) &= \sum_{i+2j=n-3k} \binom{i+k-1}{i} \binom{j+k}{j}2^{j+k}, \\
\ac_+^k(n) &= \sum_{2r+i+j=n-2k} \binom{r+k}{r} \binom{r}{i} \binom{r+j-1}{j}. \label{eq:ac1}
\end{align}
\end{theorem}

\begin{proof}[Analytic Proof]
We derive $\pc_+^k(n)$ from the generating function $\sum_{n,\, k\ge0} \pc_+^k(n) q^n t^k$.
Given a composition counted by $\pc_+^k(n)$, we pair the first part with the last part, the second part with the second last part, and so on. 
Each pair may or may not be equal.
If the middle part exists, it must be even and is not paired with any other part.      
The generating function for each pair is 
\begin{align*} 
G(q,t) &= (q^2+q^4+\cdots) + t((q+q^2+\cdots)^2-(q^2+q^4+\cdots)). 
\end{align*}
We can simplify $G(q,t)$ and obtain
\begin{align*} 
\sum_{n,\, k\ge0} \pc_+^k(n) q^n t^k 
&= \frac{1}{1-G(q,t)} \frac{1}{1-q^2} = \frac{1-q}{(1-q)(1-2q^2)-2q^3t}.
\end{align*}
Extracting the coefficient of $t^k$ gives
\begin{align*} 
\sum_{n\ge0} \pc_+^k(n) q^n 
&= \frac{(2q^3)^k}{(1-q)^k(1-2q^2)^{k+1}} 
= (2q^3)^k \sum_{i\ge0} \binom{i+k-1}{i} q^i \sum_{j\ge0} \binom{j+k}{j}(2q^2)^j.
\end{align*}
This implies the formula~\eqref{eq:pc} for $\pc_+^k(n)$.
Similarly, we obtain the generating function for $\ac_+^k(n)$:
\begin{align}\label{eq:AC}
\sum_{n,\, k\ge0} \ac_+^k(n) q^n t^k & = \frac{1}{1-tG(q,1/t)} \frac{1}{1-q^2} 
= \frac{1-q}{1-q-q^2-q^3-(1-q)q^2t}.
\end{align}
Extracting the coefficient of $t^k$ gives
\begin{align*} 
\sum_{n\ge0} \ac_+^k(n) q^n 
&= \frac{(1-q)^{k+1}q^{2k}}{(1-q-q^2-q^3)^{k+1}} 
= \frac{q^{2k}}{(1-q^2(1+q)/(1-q))^{k+1}}.
\end{align*}
Applying the binomial theorem to this gives the formula~\eqref{eq:ac1} for $\ac_+^k(n)$.
\end{proof}

\begin{proof}[Combinatorial Proof]
Let $\alpha=(\alpha_1, \ldots, \alpha_\ell)$ be a composition of $n$ such that $\alpha_{(\ell+1)/2}$ is even whenever $\ell$ is odd.
Let $S:=\{ 1\le h\le \ell/2: \alpha_h \ne \alpha_{\ell+1-h} \}$.
Define a composition $\beta_h := ( |\alpha_h-\alpha_{\ell+1-h}|: h\in S)$.
Subtracting the parts of $\beta$ from the corresponding parts of $\alpha$ gives a palindromic composition $\gamma = (\gamma_1, \ldots, \gamma_\ell)$ with $|\beta|+|\gamma|=n$, where $\gamma_h := \min\{\alpha_h, \alpha_{\ell+1-h} \}$ for all $h=1, \ldots,\ell$. 
Since $\gamma$ is palindromic, we can encode it in a binary sequence of $b=(b_1, \ldots, b_{|\gamma|/2})$, where $b_h=1$ if and only if $h\in \{\gamma_1, \gamma_1+\gamma_2, \ldots, \gamma_1+\cdots+\gamma_{\lfloor \ell/2 \rfloor}\}$.
We obtain $b'$ (resp., $b''$) from $b$ by adding $\beta_h$ to $b_{\gamma_1+\cdots+\gamma_h}$ for all $h\in S$ with $\alpha_h > \alpha_{\ell+1-h}$ (resp., $\alpha_h < \alpha_{\ell+1-h}$).
Then we have a bijection $\alpha\leftrightarrow(b',b'')$ since $\alpha$ can be recovered from $b'$ and $b''$:
\begin{itemize}
\item
Let $s$ be the number of positive entries in $b'$. 
Then $\ell$ is $2s$ if $b'$ ends in $1$ or $2s+1$ otherwise.
\item
For each $h=1, \ldots, s$, the $h$th positive entry of $b'$ (resp., $b''$) plus the number of consecutive zeros immediately before it equals $\alpha_h$ (resp.,  $\alpha_{\ell+1-h}$).
\item
If $\ell=2s+1$ then $\alpha_{s+1}$ equals twice the length of the rightmost string of consecutive zeros in $b'$.
\end{itemize} 

Restrict the bijection $\alpha\leftrightarrow(b', b'')$ to those $\alpha$ with $|S|=k$. 
Then $|\beta|=i+k$ for some $i\ge0$, $|\gamma|=2(j+k)$ for some $j\ge0$, and $n=|\beta|+|\gamma|=i+2j+3k$.
We can construct all $(b',b'')$ corresponding to those $\alpha$ with $|S|=k$ as follows.
\begin{itemize}
\item[(i)]
First let $b'$ and $b''$ be the same binary sequence with $k$ ones indexed by a set $T$ and $j$ zeros. 
\item[(ii)]	
Let $\beta=(\beta_h: h\in T)$ be a composition of $i+k$ with exactly $k$ parts indexed by $T$.
\item[(iii)]
For each $h\in T$, add $\beta_h$ to either $b'_h$ or $b''_h$.
\item[(iv)]
For each $h\in\{1, \ldots, j+k\}\setminus T$, let $b'_h=b''_h$ be either $0$ or $1$.
\end{itemize} 
The numbers of possibilities for the above steps are $\binom{j+k}{j}$, $\binom{i+k-1}{i}$, $2^k$, and $2^j$, respectively.
The formula~\eqref{eq:pc} for $\pc_+^k(n)$ follows.

Now restrict the bijection $\alpha\leftrightarrow(b', b'')$ to those $\alpha$ with $|S|=\lfloor \ell/2 \rfloor -k$, where $\ell$ is the length of $\alpha$. 
Let $i:=\#\{ h\in S: \alpha_h > \alpha_{\ell+1-h}\}$.
Then $|\beta| = i+j$ for some $j\ge0$, $|\gamma| = 2(r+k)$ for some $r\ge0$, and $n = |\beta|+|\gamma| = 2r+2k+i+j$. 
We can construct all $(b',b'')$ corresponding to those $\alpha$ with $|S|=\lfloor \ell/2 \rfloor -k$ as follows.
\begin{itemize}
\item[(i)]
First let $b'$ and $b''$ be the same binary sequence with $r$ zeros indexed by a set $T$ and $k$ ones. 
\item[(ii)]
Choose $i$ indices from $T$ without repetition and increase the corresponding entries in $b'$ and $b''$ by $2$ and $1$, respectively.
\item[(iii)]
Choose $j$ indices from $T$ with repetition and do the following each time.
\begin{itemize}
\item
If we pick one of the $i$ indices chosen at the last step, add $1$ to the corresponding entry in $b'$.
\item
If we pick an index not among the $i$ indices chosen at the last step, increase the corresponding entries in $b'$ and $b''$ by $1$ and $2$, respectively for the first time and increase only the corresponding entry in $b''$ by $1$ each time afterwards.
\end{itemize}
\end{itemize} 
The numbers of ways to finish the above steps are $\binom{r+k}{k}$, $\binom{r}{i}$, and $\binom{r+j-1}{j}$, respectively.
The formula~\eqref{eq:ac1} for $\ac_+^k(n)$ follows.
\end{proof}

We provide some examples below to illustrate the above combinatorial proof.

\begin{example}
(i) For $n=25$ and $k=3$ we have a composition $\alpha=2141124111232$ counted by $\pc_+^k(n)$ with $S=\{2,3,6\}$, $\beta=221$, $\gamma=212111 4 111212$, $b'=0110311200$, $b''=0130111100$, $i=2$, $j=7$, and $i+2j+3k=25=n$.
The construction of $(b', b'')$ begins with $b'=b''=0010100100$, which has $k=3$ ones indexed by $T=\{3,5,8\}$ and $j=7$ zeros.
Adding the second and third parts of $\beta=221$ to the second and third ones in $b'$ gives $b'=0010300200$, and adding the first part of $\beta=221$ to the first one in $b''$ gives $b''=0030100100$.
Finally, changing the zeros indexed by $2,6,7$ to ones in both $b'$ and $b''$ results in $b'=0110311200$ and $b''=0130111100$.

(ii) For $n=17$ and $k=1$ we have a composition $\alpha=2134115$ counted by $\ac_+^k(n)$ with $S=\{1,3\}$, $\beta=32$, $\gamma=2114112$, $i=1$, $j=4$, $r=5$, $b'=011300$, $b''=041100$, and $2r+2k+i+j=17=n$.
The construction of $(b',b'')$ is given below, with $T=\{1,2,4,5,6\}$.
\[ 001000 \to 001200 \to 001300 \to 011300 \to 011300 \to 011300 \]
\[ 001000 \to 001100 \to 001100 \to 021100 \to 031100 \to 041100\]

(iii) For $n=6$ and $k=0$ we have $\ac_+(6,0) = 11$ anti-palindromic compositions:
\begin{center}
\begin{tabular}{CCCCC}
r & i & j & \alpha & (b, b') \\ 
\hline
1 & 0 & 4 & 15 & (1, 5) \\
\hline
1 & 1 & 3 & 51 & (5, 1) \\
\hline
2 & 0 & 2 & 24 & (01, 03) \\
\hline
2 & 0 & 2 & 123 & (10, 30) \\
\hline
2 & 0 & 2 & 1122 & (11, 22) \\
\end{tabular} 
\qquad
\begin{tabular}{CCCCC}
r & i & j & \alpha & (b, b') \\ 
\hline
2 & 1 & 1 & 42 & (03, 01) \\
\hline
2 & 1 & 1 & 321 & (30, 10) \\
\hline
2 & 1 & 1 & 1212 & (12, 12) \\
\hline
2 & 1 & 1 & 2121 & (21, 21) \\
\hline
2 & 2 & 0 & 2211 & (22, 11) \\
\hline
3 & 0 & 0 & 6 & (000, 000) 
\end{tabular} 

\end{center}
\end{example}

We provide two more formulas for $\ac_+^k(n)$ next.

\begin{theorem}
For $n,k\ge0$ we have
\begin{align}
\ac_+^k(n) &= \sum_{2r+i+j=n-2k} 2^i \binom{r+k}{k} \binom{r}{i} \binom{i+j-1}{j} \label{eq:ac2} \\
&= \sum_{i+j+r+2s=n-2k} (-1)^i \binom{k+1}{i} \binom{j+k}{j} \binom{j}{r+s} \binom{ r+s}{r}. \label{eq:ac3} 
\end{align}
\end{theorem}

\begin{proof}
We can rewrite the generating function~\eqref{eq:AC} of $\ac_+^k(n)$ as
\begin{align*}  
\sum_{n,\, k\ge0} \ac_+^k(n) q^n t^k 
&= \frac{1}{1-q^2(2q/(1-q)+1+t)} 
= \sum_{r,k\ge0} q^{2(r+k)} \binom{r+k}{k}\binom{r}{i}\frac{(2q)^i}{(1-q)^i} t^k.
\end{align*}
Extracting the coefficient of $q^n t^k$ gives the formula~\eqref{eq:ac2}.
This formula can also be proved via the same bijection $\alpha\leftrightarrow(b',b'')$ as in the combinatorial proof of Theorem~\ref{thm:pc} but using a slightly different construction of $b'$ and $b''$:
\begin{itemize}
\item[(i)]
Let $b'$ and $b''$ be the same binary sequence with $r$ zeros indexed by a set $T$ and $k$ ones. 
\item[(ii)] 
Pick $i$ indices from $T$ without repetition and for each of them, add $1$ to the corresponding entry in $b'$ and add $2$ to the corresponding entry in $b''$, or the other way around.
\item[(iii)]
Among the $2$'s in $b'$ and $b''$, choose $j$ of them with repetition and add $1$ each time.
\end{itemize} 

We also have
\begin{align*}  
\sum_{n,\, k\ge0} \ac_+^k(n) q^n t^k 
&= \sum_{k\ge0} \frac{(1-q)^{k+1} q^{2k}t^k}{(1-q-q^2-q^3)^{k+1}} \\
&= \sum_{k\ge0} q^{2k} \sum_{i\ge0} \binom{k+1}{i} (-q)^i \sum_{j\ge0} \binom{j+k}{j} (q+q^2+q^3)^j t^k 
\end{align*}
which implies the formula~\eqref{eq:ac3}.
\end{proof}

The formula~\eqref{eq:pc} for $\pc_+^k(n)$ implies a formula for $\pc^k(n) = \pc_+^k(n) + \pc_+^k(n-1)$ by~\eqref{eq:pc+-}.
In particular, we have $\pc_+^0(n) = 2^{n/2}$ if $n$ is even or $\pc_+^0(n) = 0$ otherwise, and this implies $\pc(n) = 2^{\lfloor n/2 \rfloor}$.
We also have 
\begin{align*}  
\pc_+^1(n) &= \sum_{j=0}^{\lfloor (n-3)/2\rfloor} (j+1)2^{j+1}
= 2+(\lceil n/2 \rceil-2)2^{\lceil n/2 \rceil} 
\end{align*}
for $n\ge0$, which is essentially the same as~\cite[A036799]{OEIS}. 

On the other hand, the formulas~\eqref{eq:ac1}, \eqref{eq:ac2}, \eqref{eq:ac3} for $\ac_+^k(n)$ give formulas for $\ac^k(n) = \ac_+^k(n) + \ac_+^k(n-1)$ by~\eqref{eq:ac+-}.
For $k=0$, we have
\begin{align*}
\ac_+(n) &= \sum_{2r+i+j=n} \binom{r}{i} \binom{r+j-1}{j} 
= \sum_{2r+i+j=n} 2^i \binom{r}{i} \binom{i+j-1}{j} \\
&= \sum_{j+r+2s=n} \binom{j}{r+s} \binom{r+s}{r} - \sum_{j+r+2s=n-1} \binom{j}{r+s} \binom{ r+s}{r}
\end{align*}
which equals the tribonacci number $T'_{n+1}$ with initial conditions $T'_0=0, T'_1=1, T'_2=0$~\cite[A001590]{OEIS}.
This implies $\ac(n) = T'_{n+1}+T'_n$ and it is different from the formula $\ac(n) = T_{n}+T_{n-2}$ obtained by Andrews, Just and Simay~\cite{AntiPalComp}, although they can be derived from each other via the relations $T'_{n+1}=T_{n-1}+T_{n-2}$ and $T'_n=T_{n}-T_{n-1}$~\cite[A001590]{OEIS}.
We provide one more formula for $\ac^k(n)$ below, which reduces to another formula $\ac(n) = T_{n+1}-T_{n-1}$ when $k=0$.

\begin{theorem}
For $n,k\ge0$ we have
\begin{align}\label{eq:ac4}
\ac^k(n) &= \sum_{i+j+r+s=n-2k} (-1)^i \binom{k}{i} \binom{j+k}{j} \binom{j}{r} \binom{r}{s} - \sum_{i+j+r+s=n-2k-2} (-1)^i \binom{k}{i} \binom{j+k}{j} \binom{j}{r} \binom{r}{s}.
\end{align}
\end{theorem}

\begin{proof}
We obtain the generating function of $\ac^k(n)$ from~\eqref{eq:AC}:
\begin{align*}  
\sum_{n,\, k\ge0} \ac^k(n) q^n t^k &= \frac{1-q^2}{1-q-q^2-q^3-(1-q)q^2t} \\
&= (1-q^2) \sum_{j,k\ge0} \binom{j+k}{j} (q+q^2+q^3)^j (1-q)^k q^{2k} t^k.
\end{align*}
Extracting the coefficient of $q^n t^k$ gives the desired formula for $\ac^k(n)$.
\end{proof}

\begin{remark}
There might be a combinatorial proof for the formula~\eqref{eq:ac4} of $\ac^k(n)$ via inclusion-exclusion.
For $k=0$ this formula becomes $\ac(n)=T_{n+1}-T_{n-1}$, giving a byproduct 
\[ T_{n+1} = \sum_{j+r+s=n} \binom{j}{r} \binom{r}{s}. \]
The above formula of $T_{n+1}$ is not found in OEIS~\cite[A000073]{OEIS} but has a simple bijective proof.
In fact, it is known that the number of compositions of $n$ with no part greater than $3$ is $T_{n+1}$~\cite[A000073]{OEIS}.
Given such a composition, let $j, r, s$ be the number of parts greater than $0, 1, 2$, respectively.
Then the $r$ parts greater than $1$ are among those $j$ parts greater than $0$, the $s$ parts greater than $2$ are among those $r$ parts greater than $1$, and $j+r+s=n$.
This combinatorially proves the above formula of $T_{n+1}$.
\end{remark}

\section{Reduced partially palindromic/anti-palindromic compositions}\label{sec:rpc}

Andrews, Just and Simay~\cite{AntiPalComp} found a formula for the number $\rac(n)$ of reduced anti-palindromic compositions of $n$, where \emph{reduced anti-palindromic compositions} are equivalence classes of anti-palindromic compositions under swaps of the first and last parts, the second and second last parts, and so on.
We extend this definition and let $\rpc^k(n)$ (resp., $\rac^k(n)$) denote the number of equivalence classes of compositions counted by $\pc^k(n)$ (resp., $\ac^k(n)$) under the aforementioned swaps.
Taking $k=0$ gives $\rpc^0(n) = \rpc(n)$ and $\rac^0(n) = \rac(n)$.
We define $\rpc^k_+(n)$, $\rpc^k_-(n)$, $\rac^k_+(n)$, and $\rac^k_-(n)$ similarly.

\begin{proposition}
For $n,k\ge0$, $\rpc^k_+(n) = \pc^k_+(n)/2^k$, $\rpc^k_-(n) = \pc^k_-(n)/2^k$, and $\rpc^k(n) = \pc^k(n)/2^k$. 
\end{proposition}

\begin{proof}
This result follows immediately from the definition.
\end{proof}

\begin{theorem}\label{thm:rac}
For $n\ge0$, we have 
\[ \rac_+^k(n) = \sum_{2r+j=n-2k} \binom{r+k}{r} \binom{r+j-1}{j} \]
which is also the number of compositions of $n-k$ with exactly $k$ parts equal to $1$~\cite[A105422]{OEIS}.
\end{theorem}

\begin{proof}
Similarly to the proof of Theorem~\ref{thm:pc}, 
the generating function for $\rac_+^k(n)$ is
\begin{align*}
\sum_{n,\, k\ge0} \rac_+^k(n) q^n t^k & = \frac{1}{1-tG(q,1/2t)} \frac{1}{1-q^2} 
= \frac{1-q}{1-q-q^2-(1-q)q^2t}.
\end{align*}
Extracting the coefficient of $t^k$ gives
\begin{align*} 
\sum_{n\ge0} \rac_+^k(n) q^n 
&= \frac{(1-q)^{k+1}q^{2k}}{(1-q-q^2)^{k+1}} 
= \frac{q^{2k}}{(1-q^2/(1-q))^{k+1}}.
\end{align*}
This implies the desired formula of $\rac_+^k(n)$.
We can also prove this formula by modifying the combinatorial proof of the formula~\eqref{eq:ac2} of $\ac_+^k(n)$, where the $i$ indices picked in step (ii) correspond to the swaps of parts used to define $\rac_+^k(n)$.

Finally, comparing the generating function of $\rac_+^k(n)$ with the generating function of~\cite[A105422]{OEIS} shows that $\rac_+^k(n)$ is also the number of compositions of $n-k$ with exactly $k$ parts equal to $1$.
This can also be proved bijectively.
In fact, given a composition of $n-k$ with $k$ parts equal to $1$ and $r$ parts greater than $1$, subtracting $1$ from each part gives a composition of $n-2k-r = r+j$ for some $j\ge0$ with $r$ parts.
Conversely, if $2r+j=n-2k$ then there are $\binom{r+j-1}{j}$ many compositions of $r+j$ with $r$ parts, each corresponding to $\binom{r+k}{r}$ many compositions of $n-k$ with $k$ parts equal to $1$ and $r$ parts greater than $1$.
\end{proof}

Some special cases of Theorem~\ref{thm:rac} are given below.

\begin{itemize}
\item
We have $\rac_+^0(0)=1$ and $\rac_+^0(n) = F_{n-1}$~\cite[A212804]{OEIS}.
Consequently, $\rac(0)=1$ and $\rac(n) = F_n$ for $n\ge1$~\cite[A324969]{OEIS}.

\item
For $n\ge0$, $\rac_+^1(n) = \sum_{0\le r\le (n-2)/2} (r+1) \binom{n-r-3}{n-2r-2}$ is the number of compositions of $n-1$ with exactly one part equal to $1$~\cite[A006367]{OEIS}, and $\rac^1(n)$ is the number of compositions of $n-2$ with at most one even part~\cite[A208354]{OEIS}.

\item
For $n\ge0$, $\rac_+^2(n) = \sum_{2r+j=n-4} \binom{r+2}{2} \binom{r+j-1}{j}$ is the number of compositions of $n-2$ with exactly two parts equal to $1$~\cite[A105423]{OEIS}.
\end{itemize}

\begin{remark}
It would be interesting to see whether $\rac^1(n) = \rac^1_+(n) + \rac^1_+(n-1)$ can be proved bijectively using compositions with exactly one part equal to $1$ and compositions with at most one even part.
One may also notice that $\rac(n)$ equals the number of the number of compositions of $n$ with no even parts. 
By checking the generating functions we are convinced that the connection between $\rac^k(n)$ and compositions of $n$ with at most $k$ even parts is valid only for $k=0,1$.
\end{remark}

\section{Partially palindromic compositions modulo $m$}\label{sec:pcm}

Let $m$ be a positive integer throughout the rest of the paper. 
In this section we obtain closed formulas for $\pc_+^k(n,m)$, which imply formulas for $\pc^k(n,m)$ by~\eqref{eq:pc+-}.
\begin{theorem}\label{thm:pcm}
For $n,k\ge0$ we have
\begin{align}\label{eq:pcm1}
\pc_+^k(n,m) =& \sum_{2i+mj+(m-1)r+s = n-k} (-1)^r 2^i \binom{i}{k}\binom{i+j-1}{j}\binom{k}{r}\binom{k+s-1}{s}  \\
=& \sum_{\substack{ i_0+i_1+\cdots+i_{m-2} = k \\ 2i+mj+i_1+2i_2+\cdots+(m-2)i_{m-2} = n-k} } 
2^{i} \binom{i}{k}\binom{i+j-1}{j}\binom{k}{i_0,i_1,\ldots,i_{m-2}}. \label{eq:pcm2}
\end{align}
\end{theorem}

\begin{proof}
Given a composition, each part equals its least positive residue modulo $m$ plus a nonnegative multiple of $m$.
We pair the first part with the last part, the second part with the second last part, and so on. 
If the middle part exists, then it is not paired with any other part.  
For each pair, their smallest positive residues modulo $m$ may be equal or unequal. 
Thus the generating function for each pair is
\begin{align*}
G_m (q,t) & = (q^2+q^4+\cdots+q^{2m}+t((q+q^2+\cdots+q^m)^2-q^2-q^4-\cdots-q^{2m})(1+q^m+\cdots)^2 \\
\end{align*}
and the generating function for $\pc_+^k(n,m)$ is
\begin{align*}
\sum_{n,k\ge0} \pc_+^k(n,m) q^n t^k 
&= \frac{1}{1-G_m(q,t)}\frac{1}{1-q^2} \\
&= \frac{(1-q)(1-q^m)}{(1-q)(1-q^m)-2q^2(1-q+q(1-q^{m-1})t)} \\
&= \sum_{i\ge0} \frac{(2q^2)^i}{(1-q^m)^i} \left(1+\frac{q(1-q^{m-1})t}{1-q}\right)^i \\
&= \sum_{i\ge k\ge 0} \binom{i}{k} \frac{(2q^2)^i q^k(1-q^{m-1})^k t^k}{(1-q^m)^i(1-q)^k}.
\end{align*}
Extracting the coefficient of $q^n t^k$ gives the formula~\eqref{eq:pcm1} for $\pc_+^k(n,m)$.

We can also write $(1-q^{m-1})^k/(1-q)^k = (1+q+\cdots+q^{m-2})^k$ in the generating function of $\pc_+(n,m,k)$ and apply the multinomial theorem to obtain the formula~\eqref{eq:pcm2}.
\end{proof}

\begin{remark}
We can rewrite~\eqref{eq:pcm2} (and similar formulas appearing later) using symmetric functions.
Let $\lambda \subseteq (m-2)^k$ denote that $\lambda$ is an (integer) partition with at most $k$ parts, each no more than $m-2$.
Adding trailing zeros we can identify it with a decreasing sequence $\lambda=(\lambda_1,\ldots,\lambda_k)$ of $k$ integers in $\{0,1,\ldots,m-2\}$.
For $h=0,1,\ldots,m-2$, let $i_h$ be the number of parts of $\lambda\subseteq (m-2)^k$ that are equal to $h$.
We have $i_0+i_1+\cdots+i_{m-2}=k$ and $i_1+2i_2+\cdots+(m-2)i_{m-2}=|\lambda|$.
The \emph{monomial symmetric function} $m_\lambda(x_1,\ldots,x_k)$ is the sum of the monomials $x_1^{a_1}\cdots x_k^{a_k}$ for all rearrangements $(a_1,\ldots,a_k)$ of $(\lambda_1,\ldots,\lambda_k)$.
The formula~\eqref{eq:pcm2} becomes
\[ \pc_+^k(n,m) = \sum_{\substack{\lambda\subseteq(m-2)^k \\ 2i+mj+|\lambda|=n-k}}
2^{i} \binom{i}{k}\binom{i+j-1}{j} m_\lambda(1^k) \]
where $m_\lambda(1^k)$ is the evaluation of $m_\lambda(x_1,\ldots,x_k)$ at the vector $(1,\ldots,1)$ of length $k$.
\end{remark}

Taking $k=0$ in either~\eqref{eq:pcm1} or~\eqref{eq:pcm2} give the following: 
\begin{equation}\label{eq:pcm+}
\pc_+(n,m) = \sum_{2i+mj=n} 2^i \binom{i+j-1}{j}. 
\end{equation}
We can also obtain this by modifying the combinatorial proof of Theorem~\ref{thm:pc} with $|\beta|=mj$ and $|\gamma|=2i$.
Some known results follow. 

\begin{corollary}[{\cite{ParityPalComp,PalCompMod,RecursiveParity}; cf.~\cite[A025192]{OEIS}.}]
For $n\ge1$ we have $\pc(2n,2) = \pc(2n+1,2) = 2\cdot 3^{n-1}$.
\end{corollary}

\begin{proof}
Taking $m=2$ in~\eqref{eq:pcm+} gives $\pc_+^0(2n,2) = \sum_{i\ge0} 2^i \binom{n- 1}{n -i} = 2\cdot 3^{n-1}$ and $\pc_+^0(2n+1,2)=0$.
\end{proof}

\begin{corollary}[\cite{PalCompMod}]
(i) For $n\ge2$ we have $\pc(n,3) = 2F_{n-1}$.

\noindent
(ii) If $m$ is even then 
$\pc(2n,m)=\pc(2n+1,m)$.

\noindent
(iii) If $2n+1<m$ then $\pc(2n,m)=\pc(2n+1,m)=2^n$.
\end{corollary}

\begin{proof}
(i) The proof of Theorem~\ref{thm:pcm} gives the generating function 
\begin{align*}
\sum_{n\ge0} \pc_+^0(n,3) q^n 
&= \frac{(1-q)(1-q^3)}{(1-q)(1-q^3)-2q^2(1-q)} 
= \frac{1-q^3}{1-2q^2-q^3}.
\end{align*}
Taking $m=3$ in~\eqref{eq:pcm+} gives $\pc_+^0(n,3) = \sum_{2i+3j=n} 2^i \binom{i+j-1}{j}$. 
Either way we have $\pc_+^0(0,3)=1$, $\pc_+^0(1,3)=1$, and $\pc_+^0(n,3)=2(F_{n-2}+(-1)^{n-2})$ for $n\ge2$ (cf.~\cite[A008346]{OEIS}).
It follows from~\eqref{eq:pc+-} and the recurrence of $F_n$ that $\pc(n,3) =  \pc_+^0(n,3) + \pc_+^0(n-1,3) = 2 F_{n-1}$. 

(ii) If $m$ is even then $\pc_+^0(n,m)=0$ for all odd $n$ by~\eqref{eq:pcm+}.
Thus (ii) follows from~\eqref{eq:pc+-}.

(iii) If $2n+1<m$ then a summand in $\pc_+^0(2n,m)$ given by~\eqref{eq:pcm+} is nonzero only if $j=0$ and $i=n$, and $\pc_+^0(2n+1,m)=\pc_+^0(2n-1,m)=0$ for the same reason.
This together with~\eqref{eq:pc+-} gives (iii).
\end{proof}

For $m=1$ we have $\pc_+(n,1)=\sum_{2i+j=n} 2^i \binom{i+j-1}{j}$ and $\pc_+^k(n,1)=0$ for $k\ge1$ from
\begin{align*}
\sum_{n,k\ge0} \pc_+^k(n,1) q^n t^k 
&= \frac{1-q}{1-q-2q^2} 
= \sum_{i\ge0} \frac{(2q^2)^i}{(1-q)^i}.
\end{align*}
We can also use a combinatorial argument for this. 
In fact, since any two integers are congruent modulo $m=1$, we have must $\pc_+^k(n,1)=0$ for $k\ge1$, and for $k=0$, the number $\pc_+(n,1)$ counts compositions of $n$ without an odd middle part.
The combinatorial proof of Theorem~\ref{thm:pc} can be modified with $|\beta|=j$ and $|\gamma|=2i$ to give another proof of the above formula of $\pc_+(n,1)$.
Analogously, we have $\pc(n,1) = 2^{n-1}$ counts all compositions of $n$ and $\pc^k(n,1)=0$ for $k\ge1$.
Note that $\pc_+(n,1)$ coincide with~\cite[A078008]{OEIS}, which counts compositions of $n$ with parts greater than one, each part colored in two possible ways.
It would be interesting to have a bijection between the two families of compositions both counted by $\pc_+(n,1)$.

For $m=2$ we observe the following.

\begin{corollary}
For $n\ge0$ we have $\pc_+^k(n,2) = \sum_{2i+2j=n-k} 2^i \binom{i}{k} \binom{i+j-1}{j}$ which is zero when $n-k$ is odd.
In particular, $\pc_+^1(2n,2)=0$ and $\pc_+^1(2n+1,2) = \sum_{i\ge0} (i+1)2^{i+1}\binom{n-1}{i}$ for $n\ge1$~\cite[A081038]{OEIS}.
\end{corollary}

\begin{proof}
Taking $m=2$ in the formula~\eqref{eq:pcm2} gives $\pc_+^k(n,2) = \sum_{2i+2j=n-k} 2^i \binom{i}{k} \binom{i+j-1}{j}$.
This sum is empty if $2i+2j=n-k$ is odd.
The rest of the result follows.
\end{proof}

\section{Reduced partially palindromic compositions modulo $m$}\label{sec:rpcm}

Let $\rpc^k(n,m)$ denote the number of equivalence classes of compositions counted by $\pc^k(n,m)$ under swaps of the first and last parts, the second and second last parts, and so on.
We define $\rpc^k_+(n,m)$ and $\rpc^k_-(n,m)$ similarly.
It follows that $\rpc^k(n,m) = \rpc^k_+(n,m) + \rpc^k_-(n,m) = \rpc^k_+(n,m) + \rpc^k_+(n-1,m)$ where $\rpc^k_+(-1,m):=0$.

\begin{theorem}\label{thm:rpcm}
For $n,k\ge0$ we have
\begin{align}\label{eq:rpcm1}
\rpc_+^k(n,m) =& \sum_{2i+mj+2c+(m-1)r+s = n-k} (-1)^r \binom{i}{k}\binom{i+j-1}{j}\binom{i+c}{c}\binom{k}{r}\binom{k+s-1}{s}  \\
=& \sum_{\substack{ i_0+i_1+\cdots+i_{m-2} = k \\ 2i+mj+2c+i_1+2i_2+\cdots+(m-2)i_{m-2} = n-k} } 
\binom{i}{k}\binom{i+j-1}{j}\binom{i+c}{c}\binom{k}{i_0,i_1,\ldots,i_{m-2}}. \label{eq:rpcm2}
\end{align}
\end{theorem}

\begin{proof}
Similarly to the proof of Theorem~\ref{thm:pcm}, the generating function for each pair is
\begin{align*}
G'_m (q,t) 
=& (q^2+q^4+\cdots+q^{2m}) \left( 1+q^{2m}+\cdots +\frac12((1+q^m+\cdots)^2-1-q^{2m}-\cdots) \right) \\
&+ \frac t2 ((q+q^2+\cdots+q^m)^2-q^2-q^4-\cdots-q^{2m})(1+q^m+\cdots)^2
\end{align*}
and the generating function of $\rpc_+^k(n,m)$ is
\begin{align*}
\sum_{n,k\ge0} \rpc_+^k(n,m) q^n t^k 
&= \frac{1}{1-G'_m(q,t)}\frac{1}{1-q^2} \\
&= \frac{(1-q)(1-q^m)}{(1-q)(1-q^2)(1-q^m)-q^2(1-q) - q^3(1-q^{m-1})t} \\
&= \frac1{1-q^2} \sum_{i\ge k\ge0} \binom{i}{k}\frac{q^{2i+k}(1-q)^{i-k} (1-q^{m-1})^k t^k}{(1-q)^i(1-q^2)^i(1-q^m)^i}. \\
\end{align*}
Extracting the coefficient of $q^n t^k$ gives the desired formulas for $\rpc_+^k(n,m)$.
\end{proof}

For $k=0$ we have the number $\rpc_+(n,m)$ counts compositions $\alpha=(\alpha_1, \ldots, \alpha_\ell)$ of $n$ such that $\alpha_{(\ell+1)/2}$ is even whenever $\ell$ is odd and that $\alpha_h-\alpha_{\ell+1-h}$ is a nonnegative multiple of $m$ for $h=1, \ldots,\lfloor \ell/2 \rfloor$.
One can compare the following formula of $\rpc_+(n,m)$ with the formula~\eqref{eq:pcm+} for $\pc_+(n,m)$.

\begin{corollary}
For $n\ge0$ we have
\[ \rpc_+(n,m) = \sum_{2i+mj+2r=n} \binom{i+j-1}{j} \binom{i+r}{r}. \]
\end{corollary}
\begin{proof}
Setting $t=0$ in the generating function of $\rpc_+^k(n,m)$ gives 
\begin{align*}
\sum_{n \ge0} \rpc_+(n,m) q^n
&= \frac{(1-q^m)}{(1-q^2)(1-q^m)-q^2} 
= \frac1{1-q^2} \sum_{i\ge k\ge0} \frac{q^{2i}}{(1-q^2)^i(1-q^m)^i}. 
\end{align*}
This implies the desired formula for $\rpc_+(n,m)$, which can also be obtained by modifying the combinatorial proof of Theorem~\ref{thm:pc} with $|\beta|=mj$, $|\gamma|=2(i+r)$, and $\ell(\gamma)=2i$.
\end{proof}

In particular, $\rpc_+(n,1)$ coincides with~\cite[A052547]{OEIS} and $\rpc(n,1) = \sum_{2i+j+2r=n} \binom{i+j}{j} \binom{i+r-1}{r}$ agrees with~\cite[A028495]{OEIS}, which also counts compositions of $n$ with increments appearing only at every second position (such compositions are in bijection with the compositions counted by $\rpc(n,1)$ by reordering parts appropriately). 
Next, the generating function of $\rpc_+(n,2)$ can be written as $(F(q)-F(-q))/2q$, where $F(q) = q/(1-q-q^2)$ is the generating function of the Fibonacci numbers.
Thus $\rpc_+(2n,2)=F_{2n+1}$ and $\rpc_+(2n+1,2)=0$.
Consequently, $\rpc(n,2)$ agrees with~\cite[A094967]{OEIS}: $\rpc(2n,2)=\rpc(2n+1,2)=F_{2n+1}$.
Moreover, $\rpc_+(2n,4)$ is the same as~\cite[A052534]{OEIS} and $\rpc_+(2n+1,4)=0$.

On the other hand, for $m=1$ and $k\ge1$ we have
$\rpc_+^k(n,1) = 0$ which follows from either its definition or its generating function.
For $m=2$ we have $\rpc_+^k(n,2) = \sum_{2i+2j = n-k} \binom{i}{k}\binom{2i+j}{j}$ from its generating function
\[ \sum_{n,k\ge0} \rpc_+^k(n,2) q^n t^k 
= \frac{(1-q^2)}{(1-q^2)^2-q^2-q^3t} 
= \frac1{1-q^2} \sum_{i\ge0}\frac{q^{2i}(1+qt)^i}{(1-q^2)^{2i}}.\]
Thus $\rpc_+^1(2n,2)=0$, $\rpc_+^1(2n+1,2)= \sum_{0\le i\le n} i \binom{n+i}{2i}$~\cite[A001870]{OEIS}, and $\rpc^1(2n+1,2)=\rpc^1(2n+2,2)=\rpc_+^1(2n+1,2)$.

\section{Partially anti-palindromic compositions modulo $m$}\label{sec:acm}

In this section we provide closed formulas for $\ac_+^k(n,m)$ and $\ac^k(n,m)$, where $m$ is a positive integer.

\begin{theorem}\label{thm:acm+}
For $n,k\ge0$ we have
\begin{align*}  
\ac_+^k(n,m) &= \sum_{2i+j+r(m-1)+s+mc+md =n-2k} (-1)^r 2^j\binom{i+k}{k}\binom{i}{j}\binom{j}{r}\binom{j+s-1}{s}\binom{k}{c}\binom{k+j+d-1}{d} \\
&= \sum_{\substack{ i_0+i_1+\cdots+i_{m-2} = j \\ 2i+j+i_1+2i_2+\cdots+(m-2)i_{m-2}+mc+md = n-2k} } 
2^j\binom{i+k}{k}\binom{i}{j} \binom{j}{i_0, \ldots, i_{m-2}} \binom{k}{c}\binom{k+j+d-1}{d}. 
\end{align*}
\end{theorem}
\begin{proof}
Similarly as the proof of Theorem~\ref{thm:pcm}, we have
\begin{align*}  
\sum_{n,\, k\ge0} \ac_+^k(n,m) q^n t^k 
&= \frac{1}{1-tG_m(q,1/t)} \frac{1}{1-q^2} \\
&= \frac{(1-q)(1-q^m)}{(1-q)(1-q^m)-q^2(1-q)(1-q^m)-2q^3(1-q^{m-1})-q^2(1-q)(1+q^m)t} \\
&= \sum_{i\ge0} q^{2i} \left( 1 + \frac{2q(1-q^{m-1})}{(1-q)(1-q^m)} + \frac{(1+q^m)t}{1-q^m} \right)^i \\
&= \sum_{i,k\ge0} q^{2(i+k)} \binom{i+k}{k} \left( 1 + \frac{2q(1-q^{m-1})}{(1-q)(1-q^m)} \right)^i \frac{(1+q^m)^k t^k}{(1-q^m)^k} \\
&= \sum_{i,j,k\ge0} q^{2(i+k)} \binom{i+k}{k} \binom{i}{j} \frac{(2q)^j(1-q^{m-1})^j}{(1-q)^j(1-q^m)^j} \frac{(1+q^m)^k t^k}{(1-q^m)^k}
\end{align*}
Extracting the coefficient of $q^nt^k$ gives the first formula for $\ac_+^k(n,m)$.

We can also write $(1-q^{m-1})^j/(1-q)^j = (1+q+\cdots+q^{m-2})^j$ in the generating function of $\ac_+(n,m,k)$ and apply the multinomial theorem to obtain the second formula.
\end{proof}

Combining Theorem~\ref{thm:acm+} and the formula~\eqref{eq:ac+-} we have two formulas for $\ac^k(n,m)$.
We also directly provide another formula for $\ac^k(n,m)$ below.

\begin{theorem}\label{thm:acm}
For $n, k\ge0$ we have
\begin{align*}  
\ac^k(n,m) 
=& \sum_{3i+j+r(m-1)+2s+cm+dm=n-2k} (-1)^r 2^i \binom{i+k}{k} \binom{i+j}{j} \binom{i}{r} \binom{i+k+s-1}{s} \binom{k}{c} \binom{i+k+d-1}{d}.
\end{align*}
\end{theorem}
\begin{proof}
Similarly to the proof of Theorem~\ref{thm:pcm}, we have
\begin{align*}  
\sum_{n,\, k\ge0} \ac^k(n,m) q^n t^k 
&= \frac{1}{1-tG_m(q,1/t)} \frac{1}{1-q} \\
&= \frac{(1-q^2)(1-q^m)}{(1-q)(1-q^2)(1-q^m)-2q^3(1-q^{m-1}) - q^2(1-q)(1+q^m)t} \\
&= \frac{1}{1-q} \sum_{i\ge0} \left( \frac{2q^3(1-q^{m-1})}{(1-q)(1-q^2)(1-q^m)} + \frac{q^2(1-q)(1+q^m)t}{(1-q)(1-q^2)(1-q^m)} \right)^i \\
&= \frac{1}{1-q} \sum_{i,k\ge0} \binom{i+k}{k} \frac{(2q^3)^i (1-q^{m-1})^i}{(1-q)^i(1-q^2)^i(1-q^m)^i} \cdot \frac{q^{2k}(1+q^m)^k t^k}{(1-q^2)^k(1-q^m)^k} \\
&= \sum_{i,k\ge0} \binom{i+k}{k} \frac{2^i q^{3i+2k}(1-q^{m-1})^i(1+q^m)^k t^k} {(1-q)^{i+1}(1-q^2)^{i+k}(1-q^m)^{i+k}}.
\end{align*}
Extracting the coefficient of $q^nt^k$ from this gives the desired formula for $\ac^k(n,m)$.
\end{proof}
 
Taking $k=0$ in Theorem~\ref{thm:acm+} and Theorem~\ref{thm:acm} immediately give closed formulas for $\ac_+(n,m)$ and $\ac(n,m)$.
We do not find any connection of $\ac_+(n,m)$ and $\ac(n,m)$ with existing sequences in OEIS~\cite{OEIS}.

Another interesting specialization is at $m=1$. 
Since any two integers are congruent modulo $m=1$, the number $\ac_+^k(n,1)$ counts compositions of $n$ with $2k$ parts or with $2k+1$ parts, the middle one of which is even.
The generating function of $\ac_+^k(n,m)$ specializes to 
\[ \sum_{n,k\ge0} \ac_+^k(n,1) q^n t^k = \frac{1}{1-q^2-q^2(1+q)t/(1-q)} 
= \sum_{i,k\ge0} q^{2(i+k)} \binom{i+k}{k} \frac{(1+q)^kt^k}{(1-q)^k} \]
which implies
\[ \ac_+^k(n,1) = \sum_{2i+c+d=n-2k} \binom{i+k}{k} \binom{k}{c} \binom{k+d-1}{d}. \]
A signed version of $\ac_+^k(n,1)$ gives a Riordan array which is the coefficient table of the square of Chebyshev $S$-polynomials and also sends the Catalan numbers to ones~\cite[A158454]{OEIS}. 
We also find some special cases: $\ac_+(n,1)=(1+(-1)^n)/2$, $\ac_+^1(n,1) = \lfloor n^2/4 \rfloor$~\cite[A002620]{OEIS}, $\ac^2_+(n,1)$~\cite[A001752]{OEIS}, $\ac^3_+(n,1)$~\cite[A001769]{OEIS}, $\ac^4_+(n,1)$~\cite[A001780]{OEIS}, and $\ac^5_+(n,1)$~\cite[A001786]{OEIS}.

Similarly, the number $\ac^k(n,1)$ counts all compositions of $n$ with $2k$ or $2k+1$ parts, and
\[ \sum_{n,\, k\ge0} \ac^k(n,1) q^n t^k = \frac{1-q}{(1-q)^2-q^2t} = \sum_{k\ge0} \frac{q^{2k}t^k}{(1-q)^{2k+1}} \]
which implies $\ac^k(n,1) = \binom{2k+1+n-2k-1}{n-2k} = \binom{n}{2k}$.
This formula can also be proved combinatorially as compositions of $n$ with $2k$ or $2k+1$ parts are in bijection with binary sequences of length $n$ with exactly $2k$ ones (if the last digit is zero, change it to one to get a composition with $2k+1$ parts).

For $m=2,3$ and $k=0,1,2$ we list some initial terms of $\ac^k(n,m)$ below.
\begin{itemize}
\item
$\ac^0(n,2): 1, 1, 1, 3, 3, 7, 11, 17, 33, 49, 89, 147, 243, 423, 691, 1185, \ldots$
\item
$\ac^1(n,2): 0, 0, 1, 1, 4, 8, 13, 33, 52, 108, 201, 353, 688, 1196, 2213, 3985, \ldots$
\item
$\ac^2(n,2): 0, 0, 0, 0, 1, 1, 7, 13, 32, 80, 148, 352, 677, 1381, 2799, 5313, \ldots$
\item
$\ac^0(n,3): 1, 1, 1, 3, 5, 7, 15, 27, 43, 81, 147, 249, 449, 809, 1409, 2507, \ldots$
\item
$\ac^1(n,3): 0, 0, 1, 1, 2, 8, 13, 23, 58, 108, 195, 411, 786, 1446, 2831, 5387, \ldots$
\item
$\ac^2(n,3): 0, 0, 0, 0, 1, 1, 3, 13, 22, 48, 132, 258, 525, 1197, 2409, 4797, \ldots$
\end{itemize}
OEIS~\cite{OEIS} does not contain any of the above sequences. 

\section{Reduced partially anti-palindromic compositions modulo $m$}\label{sec:racm}
Let $\rac^k(n,m)$ denote the number of equivalence classes of compositions counted by $\ac^k(n,m)$ under swaps of the first and last parts, the second and second last parts, and so on.
Define $\rac^k_+(n,m)$ and $\rac^k_-(n,m)$ similarly.
It follows that $\rac^k(n,m) = \rac^k_+(n,m) + \rac^k_-(n,m) = \rac^k_+(n,m) + \rac^k_+(n-1,m)$ where $\rac^k_+(-1,m):=0$.
We give two closed formulas for $\rac^k_+(n,m)$ and a third formula directly for $\rac^k(n,m)$.

\begin{theorem}\label{thm:racm}
For $n,k\ge0$ we have
\begin{align*}  
\rac_+^k(n,m) &= \sum_{2i+j+r(m-1)+s+md =n-2k} (-1)^r \binom{i+k}{k}\binom{i}{j}\binom{j}{r}\binom{j+s-1}{s} \binom{k+j+d-1}{d} \\
&= \sum_{\substack{ i_0+i_1+\cdots+i_{m-2} = j \\ 2i+j+i_1+2i_2+\cdots+(m-2)i_{m-2}+md = n-2k} } 
\binom{i+k}{k}\binom{i}{j} \binom{j}{i_0, \ldots, i_{m-2}} \binom{k+j+d-1}{d},
\end{align*}
\begin{align*}  
\rac^k(n,m) 
=& \sum_{3i+j+r(m-1)+2s+dm=n-2k} (-1)^r \binom{i+k}{k} \binom{i+j}{j} \binom{i}{r} \binom{i+k+s-1}{s} \binom{i+k+d-1}{d}.
\end{align*}
\end{theorem}

\begin{proof}
Similarly to the proof of Theorem~\ref{thm:rpcm}, the generating function for $\rac_+^k(n,m)$ is
\begin{align*}
\sum_{n,\, k\ge0} \rac_+^k(n,m) q^n t^k & = \frac{1}{1-tG'_m(q,1/t)} \frac{1}{1-q^2} \\
&= \frac{(1-q)(1-q^m)}{(1-q)(1-q^m)-q^2(1-2q^m+q^{m+1})-q^2(1-q)t } \\
&= \sum_{i\ge0} q^{2i}\left(1+\frac{q(1-q^{m-1})}{(1-q)(1-q^m)} + \frac{t}{1-q^m}\right)^i \\
&= \sum_{i,k\ge0} q^{2(i+k)}\binom{i+k}{k}\left(1+\frac{q(1-q^{m-1})}{(1-q)(1-q^m)}\right)^i\frac{t^k}{(1-q^m)^k} \\
&= \sum_{i,j,k\ge0} q^{2(i+k)}\binom{i+k}{k}\binom{i}{j}\frac{q^j(1-q^{m-1})^j}{(1-q)^j(1-q^m)^j}\frac{t^k}{(1-q^m)^k}.
\end{align*}
Extracting the coefficient of $q^nt^k$ gives the desired formulas for $\rac_+^k(n,m)$.
We also have 
\begin{align*}  
\sum_{n,\, k\ge0} \rac^k(n,m) q^n t^k 
&= \frac{1}{1-tG'_m(q,1/t)} \frac{1}{1-q} \\
&= \frac{(1-q^2)(1-q^m)}{(1-q)(1-q^2)(1-q^m)-q^3(1-q^{m-1}) - q^2(1-q)t} \\
&= \frac{1}{1-q} \sum_{i\ge0} \left( \frac{q^3(1-q^{m-1})}{(1-q)(1-q^2)(1-q^m)} + \frac{q^2(1-q)t}{(1-q)(1-q^2)(1-q^m)} \right)^i \\
&= \sum_{i,k\ge0} \binom{i+k}{k} \frac{ q^{3i+2k}(1-q^{m-1})^i t^k} {(1-q)^{i+1}(1-q^2)^{i+k}(1-q^m)^{i+k}}.
\end{align*}
Extracting the coefficient of $q^nt^k$ from this gives the desired formula for $\rac^k(n,m)$.
\end{proof}

Taking $k=0$ in Theorem~\ref{thm:racm} immediately gives formulas for $\rac_+(n,m)$ and $\rac(n,m)$.
Although $\ac_+(n,m)$ and $\ac(n,m)$ are not related to any sequences in OEIS (see~Section~\ref{sec:acm}), we find the following special cases of $\rac_+(n,m)$ and $\rac(n,m)$:
\begin{itemize}
\item
We have $\rac_+(2n,1)=1$, $\rac_+(2n+1,1)=0$, $\rac(n,1)=1$ for $n\ge0$. 
\item
We have $\rac_+(0,2)=1$, $\rac_+(1,2)=0$ and for $n\ge2$ the number $\rac_+(n,2)$ counts compositions of $n-2$ with no two adjacent parts of the same parity~\cite[A062200]{OEIS}.  
\item
The sequence $\rac(n,3)$ for $n\ge0$ agrees with the sequence~\cite[A113435]{OEIS}; the latter currently does not contain any combinatorial interpretation.
\end{itemize}

For $m=1$ we obtain $\rac_+^k(n,1) = \sum_{2i+j=n-2k} \binom{i+k}{k}\binom{j+k-1}{j}$ from the generating function
\begin{align*}
\sum_{n,\, k\ge0} \rac_+^k(n,1) q^n t^k = \frac{1-q}{(1-q)-q^2(1-q)-q^2t} 
= \sum_{i,k\ge0} q^{2(i+k)}\binom{i+k}{k} \frac{t^k}{(1-q)^k}.
\end{align*}
Since every two integers are congruent modulo $m=1$, the number $\rac_+^k(n,1)$ counts all compositions $\alpha=(\alpha_1, \ldots, \alpha_\ell)$ of $n$ such that $\alpha_{(\ell+1)/2}$ is even whenever $\ell$ is odd and that $\alpha_h \ge \alpha_{\ell+1-h}$ for $h=1, \ldots, \lfloor \ell/2 \rfloor$. 
Thus the above formula for $\rac_+^k(n,1)$ can also be obtained by modifying the combinatorial proof of Theorem~\ref{thm:pc} with $|\beta|=j$, $|\gamma|=2(i+k)$, and $\lfloor \ell(\gamma)/2 \rfloor =k$.
We have $\rac_+^1(2n,1) = \rac_+^1(2n+1,1) = \sum_{0\le i\le n-1} (i+1) = n(n+1)/2$, giving a new interpretation for the sequence of repeated triangular numbers~\cite[A008805]{OEIS}.
We also have $\rac_+^2(n,1) = \sum_{2i+j=n-4} \binom{i+2}{2} (j+1)$~\cite[A096338]{OEIS}.

Similarly, we obtain $\rac^k(n,1) = \sum_{2i+j=n-2k} \binom{i+k-1}{i}\binom{j+k}{j}$ from the generating function
\begin{align*}
\sum_{n,\, k\ge0} \rac^k(n,1) q^n t^k = \frac{1-q^2}{(1-q)(1-q^2)-q^2t} 
= \frac{1}{1-q}\sum_{i,k\ge0} \frac{q^{2k}t^k}{(1-q)^k(1-q^2)^k}.
\end{align*}
The number $\rac^k(n,1)$ counts compositions $\alpha=(\alpha_1, \ldots, \alpha_\ell)$ of $n$ such that $\alpha_h \ge \alpha_{\ell+1-h}$ for $h=1, \ldots, \lfloor \ell/2 \rfloor$. 
This gives a new interpretation for a triangular array of integers~\cite[A060098]{OEIS}.
Some special cases include $\rac^1(n,1) = \sum_{0\le i\le \lfloor n/2 \rfloor-1} (n-2i-1) = \lfloor \frac n 2 \rfloor \lceil \frac n 2 \rceil$ which coincides with both $\ac_+^1(n,1)$ and~\cite[A002620]{OEIS}, $\rac^2(n,1) = \sum_{2i+j=n-4} (i+1)\binom{j+2}{2}$~\cite[A002624]{OEIS}, $\rac^3(n,1)$~\cite[A060099]{OEIS}, $\rac^4(n,1)$~\cite[A060100]{OEIS}, and $\rac^5(n,1)$~\cite[A060101]{OEIS}.
A combinatorial explanation for $\rac^1(n,1) = \ac_+^1(n,1)$ would be interesting.

For $m\ge2$ we do not find any connection of $\rac_+^k(n,1)$ or $\rac^k(n,1)$ with existing sequences in OEIS~\cite{OEIS}.

\section*{Acknowledgment}

The author uses SageMath to help discover and verify the closed formulas in this paper.

\end{document}